\documentclass{article}
\usepackage[all]{xy}
\usepackage{amsmath}
\usepackage{amsfonts}
\usepackage{amssymb}
\usepackage{mathrsfs}
\usepackage{amscd}
\usepackage{pb-diagram}
\usepackage{amsthm}
\usepackage{color}
\usepackage{graphicx}
\usepackage{url}

\begin{document}

\newtheorem{lemma}{Lemma}[section]
\newtheorem{definition}[lemma]{Definition}
\newtheorem{proposition}[lemma]{Proposition}
\newtheorem{corollary}[lemma]{Corollary}
\newtheorem{theorem}[lemma]{Theorem}
\newtheorem{remark}[lemma]{Remark}
\newtheorem{example}[lemma]{Example}

\markboth{Mathieu Molitor}
{Remarks on the statistical origin of the geometrical formulation of quantum mechanics}

\title{Remarks on the statistical origin of the geometrical formulation of quantum mechanics}

\author{Mathieu Molitor\thanks{Present address: \it{Fakult\"{a}t f\"{u}r Mathematik, Ruhr-Universit\"{a}t Bochum, Germany}}\\
\it \small{Department of Mathematics, Keio University}\\
\it \small{3-14-1, Hiyoshi, Kohoku-ku, 223-8522, Yokohama, Japan}\\ 
\small{\it{e-mail:}}\,\,\url{pergame.mathieu@gmail.com}
}
%	\footnotetext[*]{Present adress: Ruhr-Universit\"{a}t Bochum\\
%	Fakult\"{a}t f\"{u}r Mathematik\\
%	NA 4/32, Universit\"{a}tsstr. 150\\
%	44780 Bochum (Germany)}
\date{}
\maketitle

\begin{abstract}
	A quantum system can be entirely described by the K\"{a}hler structure of the projective 
	space $\mathbb{P}(\mathcal{H})$ associated to the Hilbert space $\mathcal{H}$ of possible states; 
	this is the so-called geometrical formulation of quantum mechanics. 
	
	In this paper, we give an explicit link between the geometrical formulation (of finite dimensional 
	quantum systems) and statistics through the natural geometry of the space $\mathcal{P}_{n}^{\times}$ 
	of non-vanishing probabilities $p\,:\,E_{n}\rightarrow \mathbb{R}_{+}^{*}$ defined on a finite set 
	$E_{n}:=\{x_{1},...,x_{n}\}\,.$ More precisely, we use the 
	Fisher metric $g_{F}$ and the exponential connection $\nabla^{(1)}$ (both being 
	natural statistical objects living on $\mathcal{P}_{n}^{\times}$) to construct, 
	via the Dombrowski splitting Theorem, a K\"{a}hler structure
	on $T\mathcal{P}_{n}^{\times}$ which has the property that it induces the natural 
	K\"{a}hler structure of a suitably chosen open dense subset of $\mathbb{P}(\mathbb{C}^{n})\,.$ 

	As a direct physical consequence, a significant part of the quantum mechanical formalism (in finite 
	dimension) is encoded in the triple $(\mathcal{P}_{n}^{\times},g_{F},\nabla^{(1)})\,.$
\end{abstract}

%\keywords{geometrical formulation of quantum mechanics; information geometry; Fisher metric; exponential connection; complex projective space;
%K\"{a}hler structure.}

\section{Introduction}
	
	In quantum mechanics, it is well known that if $\psi_{1}\,,\psi_{2}$ are two collinear vectors 
	belonging to the Hilbert space $\mathcal{H}$ describing 
	the possible states of a quantum system, then they are physically equivalent. 
	This means that the true configuration space of quantum mechanics 
	is the projective space $\mathbb{P}(\mathcal{H})\,,$ and this leads to a formulation of quantum 
	mechanics uniquely based on the geometry of $\mathbb{P}(\mathcal{H})$ 
	(see \cite{Ashtekar} for a detailed discussion). In this 
	formulation, the dynamics is governed by the Fubini-Study symplectic form, 
	and the observables are no more 
	operators, rather functions on the projective space having the particularity to preserve,
	in some sense, the Fubini-Study metric; eigenstates are
	critical points of the observable functions, and the corresponding critical values 
	are the eigenvalues. In other words, quantum mechanics can be
	completely formulated in terms of the {K\"{a}hler structure} of the projective 
	space $\mathbb{P}(\mathcal{H})\,.$
	Subsequently, we shall refer to this formulation as the 
	\textit{geometrical formulation} of quantum mechanics.
	
	The geometrical formulation, although complete and very elegant, has a very disappointing feature : 
	it doesn't explain, nor even justify, why its probabilistic interpretation, based on 
	the Fubini-Study metric, is 
	consistent. Consistency is only something that one observes once the computations are done, and,
	somehow, this appears as ``magical". But of course, magic doesn't exist, and this is clearly 
	the sign that there is something 
	in the geometrical formulation that we still don't understand. This concerns the various 
	formulas where probabilities are expected of course, but more generally, it concerns 
	our understanding of the link between the K\"{a}hler structure of the projective 
	space $\mathbb{P}(\mathcal{H})\,,$ and statistics. 

	This link has, until now, attracted very little attention in the existing literature, 
	and its nature is still quite mysterious. The only mention (to our knowledge) of a 
	possible link between the K\"{a}hler structure of the projective space,
	and statistics, comes from \textit{information geometry}\footnote{Information geometry 
	is the branch of mathematics devoted to the study of statistics using differential 
	geometrical tools (see \cite{Amari-Nagaoka}).}, where it is argued that 
	the Fubini-Study metric (at least for $\mathbb{P}(\mathbb{C}^{n})$), is a kind 
	of ``quantum analogue" of the famous \textit{Fisher metric}, the latter being a metric living 
	on the space $\mathcal{P}_{n}^{\times}$ of non-vanishing probabilities 
	$p\,:\,E_{n}\rightarrow \mathbb{R}$ defined on a finite 
	set $E_{n}=\{x_{1},...,x_{n}\}\,.$ 
	Such ``link" is, however, only an analogy, and in particular, 
	it has no precise mathematical formulation. 
	
	Nevertheless, the idea that the K\"{a}hler structure 
	of $\mathbb{P}(\mathbb{C}^{n})$ may originate from the intrinsic geometry of a
	{statistical manifold}\footnote{A statistical manifold $S$ 
	is a manifold having the property that to every point $x\in S$ is associated (in an injective way) 
	a probability $p$
	on a fixed measurable space $\Omega\,.$}, is an appealing one, and
	the main observation of this paper is that 
	it is ``almost" true. More precisely, we show that the restriction of the K\"{a}hler structure of 
	$\mathbb{P}(\mathbb{C}^{n})$ to the open dense subset\footnote{In the definition of 
	$\mathbb{P}(\mathbb{C}^{n})^{\times}$ given above, we use homogeneous coordinates $[z_{1},...,z_{n}]:=
	\mathbb{C}\cdot z\subseteq \mathbb{C}^{n}\,,$ where $z=(z_{1},...,z_{n})\in \mathbb{C}^{n}-\{0\}\,.$}
	$\mathbb{P}(\mathbb{C}^{n})^{\times}:=\{[z_{1},...,z_{n}]\in 
	\mathbb{P}(\mathbb{C}^{n})\,\vert\,
	z_{k}\neq 0\,\,\forall\, k=1,...,n\}$ 
	can be completely recovered from the statistical manifold 
	$\mathcal{P}_{n}^{\times}\,,$ supplemented with the Fisher metric $g_{F}$ {and} from another 
	important statistical object living naturally on $\mathcal{P}_{n}^{\times}\,,$ namely the 
	\textit{exponential connection} $\nabla^{(1)}\,.$ 
	Together, these geometrical objects form a triple $(\mathcal{P}_{n}^{\times},g_{F},\nabla^{(1)})$ 
	which allows the construction of an almost 
	Hermitian structure $(T\mathcal{P}_{n}^{\times},G,J)$ on $T\mathcal{P}_{n}^{\times}\,;$ 
	the construction, which is due to Dombrowski (see \cite{Dombrowski}) is straightforward.
	For $u_{p}\in T_{p}\mathcal{P}_{n}^{\times}\,,$ the exponential connection 
	$\nabla^{(1)}$ gives a splitting of $T_{u_{p}}T\mathcal{P}_{n}^{\times}$ into 
	the direct sum of the spaces of horizontal and vertical tangent vectors 
	that one may identify with $T_{p}\mathcal{P}_{n}^{\times}\oplus 
	T_{p}\mathcal{P}_{n}^{\times}\,.$ With this splitting, the metric $G$ at the point $u_{p}$
	is simply the direct sum $(g_{F})_{p}\oplus (g_{F})_{p}\,,$ 
	while the almost complex structure $J$ is given, for 
	$v_{p},w_{p}\in T_{p}\mathcal{P}_{n}^{\times}\,,$ 
	by $J_{u_{p}}(v_{p},w_{p}):=(-w_{p},v_{p})\,.$  
	The resulting almost Hermitian structure will be referred to as the \textit{almost 
	splitting Hermitian structure} associated to the triple 
	$(\mathcal{P}_{n}^{\times},g_{F},\nabla^{(1)})\,.$ 
	
	Now, the link between the almost splitting Hermitian structure associated to the triple 
	$(\mathcal{P}_{n}^{\times},g_{F},\nabla^{(1)})$ and the K\"{a}hler structure of 
	$\mathbb{P}(\mathbb{C}^{n})^{\times}$ goes as follows. As we show, there exists 
	a covering map $\tau\,:\,T\mathcal{P}_{n}^{\times}\rightarrow 
	\mathbb{P}(\mathbb{C}^{n})^{\times}$ having the property that the ``pull back" of the K\"{a}hler
	structure of $\mathbb{P}(\mathbb{C}^{n})^{\times}$ is exactly 
	the above almost splitting Hermitian structure 
	on $T\mathcal{P}_{n}^{\times}$ (the latter being actually also a K\"{a}hler structure). 
	Hence, we have a precise description of the statistical origin of the 
	K\"{a}hler structure of $\mathbb{P}(\mathbb{C}^{n})^{\times}\,.$ 

	A direct physical consequence --at least theoretically-- is that a significant part of the 
	quantum mechanical formalism (in finite dimension) is encoded in the triple 
	$(\mathcal{P}_{n}^{\times},g_{F},\nabla^{(1)})$ 
	via its associated almost splitting Hermitian structure. 
	This physical consequence is, however, not discussed in this paper for which we 
	refer the reader to \cite{MolitorPreprint}. \\

	This paper is organized as follows. 
	In \S\ref{section The geometry of}, we recall the definition of the Fisher 
	metric $g_{F}$ and the exponential connection 
	$\nabla^{(1)}$ on the space of probabilities $\mathcal{P}_{n}^{\times}\,,$ and we define 
	the almost splitting Hermitian structure associated to 
	$(\mathcal{P}_{n}^{\times},g_{F},\nabla^{(1)})\,.$ In \S\ref{section statistical nature}, we 
	introduce a covering map $\tau\,:\,T\mathcal{P}_{n}^{\times}\rightarrow 
	\mathbb{P}(\mathbb{C}^{n})^{\times}$ and show that the ``pull back" of the K\"{a}hler 
	structure of $\mathbb{P}(\mathbb{C}^{n})^{\times}$ 
	via $\tau$ is exactly the almost splitting Hermitian structure on 
	$T\mathcal{P}_{n}^{\times}$ associated to 
	the triple $(\mathcal{P}_{n}^{\times},g_{F},\nabla^{(1)})$ 
	(Proposition \ref{proposition origine des stat}).
\section{Information geometry and the statistical model $\mathcal{P}_{n}^{\times}$}
	\label{section The geometry of}
	Our statistical model will be the space $\mathcal{P}_{n}^{\times}$ 
	of non-vanishing probability distributions
	$p$ on a discrete set $E_{n}:=\{x_{1},...,x_{n}\}\,:$
	\begin{eqnarray}
		\mathcal{P}_{n}^{\times}:=\Big\{p\,:\,E_{n}\rightarrow \mathbb{R}\,\,\,
		\Big\vert\,\,\,p(x_{i})>0\,\,\,\text{for all}\,\,\,
		x_{i}\in E_{n}\,\,\,\text{and}\,\,\,\sum_{i=1}^{n}p(x_{i})=1\Big\}\,.
	\end{eqnarray}
	This space is clearly a connected manifold of dimension $n-1\,.$ From a 
	topological point of view, it is trivial since it can be realized as the interior of 
	a simplex, but from a statistical point of view, it is naturally endowed 
	with nontrivial geometric structures that we now want to describe (see \cite{Amari-Nagaoka}).\\

	First of all, and for the rest of this paper, we will always use the so-called 
	exponential representation for the tangent space :
	\begin{eqnarray}
		T_{p}\mathcal{P}_{n}^{\times}\cong \{u=(u_{1},...,u_{n})\in \mathbb{R}^{n}\,
		\vert\,u_{1}\,p_{1}+\ldots + u_{n}\,p_{n}=0\}\,,
	\end{eqnarray}
	where $p\in \mathcal{P}_{n}^{\times}\,,$ and where by definition, $p_{i}:=p(x_{i})$ 
	for all $x_{i}\in E_{n}\,.$\\
	If $u\in \mathbb{R}^{n}$ is a vector satisfying  $u_{1}\,p_{1}+\ldots+u_{n}\,p_{n}=0$ 
	for a given probability $p\,,$ then we shall denote by $[u]_{p}$ the unique tangent vector 
	of $\mathcal{P}_{n}^{\times}$ at the point $p$ determined by the exponential representation.
	One easily sees that if $p(t)$ is a smooth curve in $\mathcal{P}_{n}^{\times}\,,$ 
	then
	\begin{eqnarray}\label{equation derivée plus identification expo}
		\frac{d}{dt}\bigg\vert_{0}p(t)=[u]_{p(0)}\,\,\,\,\,\Leftrightarrow\,\,\,\,\,
		\dfrac{d}{dt}\bigg\vert_{0}\,p_{i}(t)=p_{i}(0)\,u_{i}
		\,\,\,\text{for all}\,\,i=1,\cdots,n\,,
	\end{eqnarray}
	where $p_{i}(t):=\big(p(t)\big)(x_{i})\,.$\\
	Equation \eqref{equation derivée plus identification expo} is actually one way 
	to define the exponential representation.\\ 
	
	The Fisher metric $g_{F}$ is now defined, for $[u]_{p},[v]_{p}
	\in T_{p}\mathcal{P}_{n}^{\times}\,,$
	by\footnote{The Fisher metric is actually defined only up to a multiplicative constant, 
	and thus we are free, for later convenience, to introduce the factor 1/4.} 
	\begin{eqnarray}
		(g_{F})_{p}([u]_{p},[v]_{p}):=\dfrac{1}{4}\sum_{k=1}^{n}\,p_{k}u_{k}v_{k}=\dfrac{1}{4}E_{p}(uv)\,,
		\end{eqnarray}
	where $E_{p}(.)$ denotes the expectation with respect to the probability $p$ and 
	where $uv$ denotes the vector in $\mathbb{R}^{n}$ whose $k$-th component is $u_{k}v_{k}\,.$\\
	In most interesting 
	statistical models, the Fisher metric is naturally and intrinsically defined, and is 
	of central importance in information geometry (see \cite{Amari-Nagaoka}).\\
	
	For a real parameter $\alpha\,,$ we also introduce the so-called $\alpha-$connection 
	$\nabla^{(\alpha)}$ on the space $\mathcal{P}_{n}^{\times}\,.$ One way to define it 
	is to use its associated covariant derivative along curves $D^{(\alpha)}/dt\,:$
	if $\big[V(t)\big]_{p(t)}$ is a vector field along a curve $p(t)$ in 
	$\mathcal{P}_{n}^{\times}$ such that $dp(t)/dt=[u(t)]_{p(t)}\,,$
	then, by definition, 
	\begin{eqnarray}\label{equation definition derivée covariante}
		&&\dfrac{D^{(\alpha)}}{dt}\big[V(t)\big]_{p(t)}:=\nonumber\\
		&&\Big[\dot{V}(t)+(1-\alpha)/2\cdot u(t)V(t)
		-E_{p(t)}\Big(\dot{V}(t)+(1-\alpha)/2\cdot u(t)V(t)\Big)\cdot n\Big]_{p(t)}\,,
	\end{eqnarray}
	where $u$ and $V$ are viewed as maps $\mathbb{R}\rightarrow \mathbb{R}^{n}\,,$ 
	$\dot{V}(t)$ is the usual derivative of $V(t)$ with respect to $t$
	and where $n:=(1,...,1)\in \mathbb{R}^{n}\,.$\\
	It may be shown that $\nabla^{(0)}$ corresponds to the Levi-Civita connection associated 
	to the Fisher metric $g_{F}$ and also, if $X,Y,Z$ are vector fields on 
	$\mathcal{P}_{n}^{\times}\,,$ that 
	\begin{eqnarray}\label{equation dual connections}
		X\,g_{F}(Y,Z)=g_{F}\big(\nabla^{(\alpha)}_{X}Y,Z\big)
		+g_{F}\big(Y,\nabla_{X}^{(-\alpha)}Z\big)\,.
	\end{eqnarray}
	Because of \eqref{equation dual connections}, the triple 
	$\big(g_{F},\nabla^{(\alpha)},\nabla^{(-\alpha)}\big)$ is called a dualistic structure on 
	$\mathcal{P}_{n}^{\times}\,,$ and is also one of the major tools in information geometry 
	(see \cite{Amari-Nagaoka}).\\

	In the sequel, we will not use this dualistic structure, 
	but we will restrict our attention to the 1-connection $\nabla^{(1)}\,,$ also called
	exponential connection. This connection is probably,
	in view of \eqref{equation definition derivée covariante}, the simplest 
	and the most natural connection among the family of $\alpha$-connections since 
	its expression reduces to 
	\begin{eqnarray}\label{equation connection exp}
		\dfrac{D^{(1)}}{dt}\big[V(t)\big]_{p(t)}=
		\big[\,\dot{V}(t)-E_{p(t)}\big(\dot{V}(t)\big)\cdot n\,\big]_{p(t)}\,.
	\end{eqnarray}

	We now want to describe the natural geometry of $T\mathcal{P}_{n}^{\times}\,.$ 
	Recall that if $M$ is a manifold endowed with an affine connection $\nabla\,,$ then the 
	Dombrowski splitting Theorem holds (see \cite{Dombrowski}) :
	\begin{eqnarray}
		T(TM)\cong TM\oplus TM\oplus TM\,,
	\end{eqnarray}
	this splitting being viewed as an isomorphism of vector bundles over $M\,,$ and the 
	isomorphism, say $\Phi_{M}\,,$
	being
	\begin{eqnarray}\label{equation Dombrowski}
		T_{u_{x}}TM\ni A_{u_{x}}\overset{\Phi_{M}}{\longmapsto} 
		\big(u_{x},(\pi^{TM})_{*_{u_{x}}}A_{u_{x}},K A_{u_{x}}\big)\,,
	\end{eqnarray}
	where $\pi^{TM}\,:\,TM\rightarrow M$ is the canonical projection and where 
	$K\,:\,T(TM)\rightarrow TM$ is the canonical connector associated to the connection 
	$\nabla\,.$ \\
	Applied to the couple $\big(\mathcal{P}_{n}^{\times},\nabla^{(1)}\big)\,,$ the 
	Dombrowski splitting Theorem yields 
	\begin{eqnarray}\label{equation dombrowski proba}
		T(T\mathcal{P}_{n}^{\times})\cong T\mathcal{P}_{n}^{\times}\oplus T\mathcal{P}_{n}^{\times}
		\oplus T\mathcal{P}_{n}^{\times}\,.
	\end{eqnarray}
	It turns out that the inverse of the vector bundle isomorphism 
	$\Phi_{\mathcal{P}_{n}^{\times}}\,:\,T(T\mathcal{P}_{n}^{\times})\rightarrow  
	T\mathcal{P}_{n}^{\times}\oplus T\mathcal{P}_{n}^{\times}
	\oplus T\mathcal{P}_{n}^{\times}$ can be expressed very explicitly :
\begin{lemma}\label{lemme formule inverse identification}
	For $[u]_{p},[v]_{p},[w]_{p}
	\in T_{p}\mathcal{P}_{n}^{\times}\,,$ we have : 
	\begin{eqnarray}
		\Phi_{\mathcal{P}_{n}^{\times}}^{-1}\big([u]_{p},[v]_{p},[w]_{p}\big) = 
		\dfrac{d}{dt}\bigg\vert_{0}\,
		\Big[u+tw-E_{p(t)}(u+tw)\cdot n\Big]_{p(t)}\,,
	\end{eqnarray}
	where $p(t)$ is a smooth curve in $\mathcal{P}_{n}^{\times}$ satisfying $p(0)=p$ and 
	$dp(t)/dt\big\vert_{0}=[v]_{p}\,.$
\end{lemma}
\begin{proof}
	It suffices to show that 
	\begin{eqnarray}
		\Phi_{\mathcal{P}_{n}^{\times}}\bigg(\dfrac{d}{dt}\bigg\vert_{0}\,
		\Big[u+tw-E_{p(t)}(u+tw)\cdot n\Big]_{p(t)}\bigg)=\big([u]_{p},[v]_{p},[w]_{p}\big)\,.
	\end{eqnarray}
	To this end, let us consider the curve $\gamma$ in $T\mathcal{P}_{n}^{\times}$ which is 
	defined by 
	\begin{eqnarray}
		\gamma(t):=\Big[u+tw-E_{p(t)}(u+tw)\cdot n\Big]_{p(t)}\,.
	\end{eqnarray}
	We have	
	\begin{eqnarray}
		&\bullet&\gamma(0)=[u]_{p}\,,\\
		&\bullet&
			(\pi^{T\mathcal{P}_{n}^{\times}})_{*_{[u]_{p}}}
			\dfrac{d}{dt}\bigg\vert_{0}\gamma(t)=
			\frac{d}{dt}\bigg\vert_{0}(\pi^{T\mathcal{P}_{n}^{\times}}\circ \gamma)(t)=
			\frac{d}{dt}\bigg\vert_{0}p(t)=[v]_{p}\,,
	\end{eqnarray}
	and thus, we immediately see that 
	\begin{eqnarray}
		\Phi_{\mathcal{P}_{n}^{\times}}\bigg(\dfrac{d}{dt}\bigg\vert_{0}\gamma(t)\bigg)=
		\big([u]_{p},[v]_{p},\ast\big)\,,
	\end{eqnarray}
	where $``\ast"$ has to be determined. But, according to \eqref{equation connection exp},
	\eqref{equation Dombrowski} and the fact that $K^{(1)}$ is a connector (here $K^{(1)}$ 
	is the connector associated to the connection $\nabla^{(1)}$), 
	\begin{eqnarray}
		\ast &=& K^{(1)}\dfrac{d}{dt}\bigg\vert_{0}\gamma(t)
			= \frac{D^{(1)}}{dt}\bigg\vert_{0}\big(\gamma(t)\big)=
			\Bigg[\dfrac{d}{dt}\bigg\vert_{0}\Big(u+tw-E_{p(t)}(u+tw)\cdot n\Big)\nonumber\\
		&&	-E_{p}\bigg(\dfrac{d}{dt}\bigg\vert_{0}\Big(u+tw-E_{p(t)}(u+tw)\cdot n\Big)
			\bigg)\cdot n\Bigg]_{p}\nonumber\\
		&=&\Bigg[w-\dfrac{d}{dt}\bigg\vert_{0}\,E_{p(t)}(u+tw)\cdot n
			+\bigg(\dfrac{d}{dt}\bigg\vert_{0}\,E_{p(t)}(u+tw)\bigg)
			\cdot E_{p}(n)\cdot n\Bigg]_{p}=[w]_{p}\,.\,\,\,\,\,\,\,\,\,\textbf{}
	\end{eqnarray}
	Hence $\Phi_{\mathcal{P}_{n}^{\times}}\Big(d\gamma(t)/dt\big\vert_{0}\Big)=
	\big([u]_{p},[v]_{p},[w]_{p}\big)\,.$ The lemma follows. 
\end{proof}
	In the sequel, we will identify a triple $\big([u]_{p},[v]_{p},[w]_{p}\big)$ 
	with the corresponding element of $T_{[u]_{p}}(T\mathcal{P}_{n}^{\times})$ 
	via the isomorphism $\Phi_{\mathcal{P}_{n}^{\times}}\,.$	\\
	
	Having the decomposition \eqref{equation dombrowski proba}, it is a simple matter 
	to define on $T\mathcal{P}_{n}^{\times}$ an almost Hermitian structure. 
	Indeed, we define a metric $G\,,$ a $2$-form $\Omega$ and an almost complex 
	structure $J$ by setting
	\begin{eqnarray}\label{equation definition G, omega, etc.}
		G_{[u]_{p}}\Big(\big([u]_{p},[v]_{p},[w]_{p}\big),
		\big([{u}]_{p},[\overline{v}]_{p},
		[\overline{w}]_{p}\big)\Big)&:=&
		(g_{F})_{p}\big([v]_{p},[\overline{v}]_{p}\big)+
		(g_{F})_{p}\big([w]_{p},[\overline{w}]_{p}\big)\,,\nonumber\\
	\Omega_{[u]_{p}}\Big(\big([u]_{p},[v]_{p},[w]_{p}\big),
		\big([{u}]_{p},[\overline{v}]_{p},
		[\overline{w}]_{p}\big)\Big)&:=&(g_{F})_{p}\big([v]_{p},[\overline{w}]_{p}\big)-
		(g_{F})_{p}\big([w]_{p},[\overline{v}]_{p}\big)\,,\nonumber\\
	J_{[u]_{p}}\Big(\big([u]_{p},[v]_{p},[w]_{p}\big)\Big)&:=&
		\big([u]_{p},-[w]_{p},[v]_{p}\big)\,,
	\end{eqnarray}
	where $[u]_{p},[v]_{p},[w]_{p},[\overline{v}]_{p},[\overline{w}]_{p}
	\in T_{p}\mathcal{P}_{n}^{\times}\,.$\\

	Clearly, $J^{2}=-\textup{Id}$  and $G(J\,.\,,J\,.\,)=G(\,.\,,\,.\,)\,,$ which means that 
	$(T\mathcal{P}_{n}^{\times},G,J)$ is an almost Hermitian manifold, and one readily sees that 
	$G,J$ and $\Omega$ are compatible, i.e., that 
	$\Omega=G\big(J\,.\,,\,.\,\big)\,;$ the $2$-form $\Omega$ is thus the fundamental 2-form of 
	the almost Hermitian manifold $(T\mathcal{P}_{n}^{\times},G,J)\,.$ 
	
	The above geometric construction is a particular case of a more general construction which is due to 
	Dombrowski (see \cite{Dombrowski}). 

\section{The statistical nature of $\mathbb{P}(\mathbb{C}^{n})$}
\label{section statistical nature}
	
	Recall that the projective space $\mathbb{P}(\mathbb{C}^{n})$ is simply the quotient 
	$(\mathbb{C}^{n}-\{0\})/{\sim}\,,$ where the equivalence relation $``\sim"$ 
	is defined by 
	\begin{eqnarray}
		(z_{1},...,z_{n})\sim (w_{1},...,w_{n})\,\,\,\,\,\Leftrightarrow\,\,\,\,\,
		\exists \lambda\in \mathbb{C}-\{0\}: (z_{1},...,z_{n})=\lambda (w_{1},...,w_{n})\,.
	\end{eqnarray}
	For $z=(z_{1},...,z_{n})\in \mathbb{C}^{n}-\{0\}\,,$ we shall denote by $[z]=[z_{1},...,z_{n}]$
	the corresponding element of $\mathbb{P}(\mathbb{C}^{n})\,.$ One may 
	identify $[z]$ with the complex line $\mathbb{C}\cdot z\,.$\\

	The manifold structure of $\mathbb{P}(\mathbb{C}^{n})$ may be defined as follows. 
	For a vector $u=(u_{1},...,u_{n})\in \mathbb{C}^{n}$ such that 
	$|u|^{2}=\langle u,u\rangle=\overline{u}_{1}u_{1}+\cdots +\overline{u}_{n}u_{n}=1$ 
	(our convention for the Hermitian product $\langle\,,\,\rangle$ on $\mathbb{C}^{n}$
	is that $\langle\,,\,\rangle$ is linear in the second argument)\,,
	we define a chart $(U_{u},\phi_{u})$ of $\mathbb{P}(\mathbb{C}^{n})$ by letting
	\begin{eqnarray}\label{equation definition carte projective}
		\left \lbrace
			\begin{array}{cc}
				U_{u}:=\big\{[z]\in \mathbb{P}(\mathbb{C}^{n})\,\big\vert\, 
				[u]\cap [z]=\{0\}\big\}\,,\\
				\phi_{u}\,:\,U_{u}\rightarrow [u]^{\perp}\subseteq \mathbb{C}^{n}\,,
				[z]\mapsto\dfrac{1}{\langle u,z\rangle}\cdot z-u\,.
			\end{array}
		\right.
	\end{eqnarray}
	If $u$ varies among all the unit vectors in $\mathbb{C}^{n}\,,$ then the corresponding 
	charts $(U_{u},\phi_{u})$ form an atlas for $\mathbb{P}(\mathbb{C}^{n})\,;$ the projective 
	space is thus a real manifold of dimension $2(n-1)\,,$ and, using the above charts, 
	we have the identification
	\begin{eqnarray}
		T_{[u]}\mathbb{P}(\mathbb{C}^{n})\cong [u]^{\perp}=\{w\in\mathbb{C}^{n}\,\big\vert\,
		\langle u,w\rangle=0\}\,.
	\end{eqnarray}
	
	The Fubini-Study metric $g_{FS}$ and the Fubini-Study symplectic form $\omega_{FS}$ 
	are now defined at the point $[u]$ in $\mathbb{P}(\mathbb{C}^{n})$ via the formulas : 
	\begin{eqnarray}\label{equation definition de fubini-study}
		\Big((\phi_{u}^{-1})^{*}g_{FS}\Big)_{0}(\xi_{1},\xi_{2}):=\textup{Re}\,
		\langle\xi_{1},\xi_{2}\rangle\,\,,\,\,\,\,\,\,
		\Big((\phi_{u}^{-1})^{*}\omega_{FS}\Big)_{0}(\xi_{1},\xi_{2}):=\textup{Im}\,
		\langle\xi_{1},\xi_{2}\rangle\,,
	\end{eqnarray}
	where $\xi_{1},\xi_{2}\in [u]^{\perp}\cong T_{[u]}\mathbb{P}(\mathbb{C}^{n})\,.$\\
	One may show that $g_{FS}$ and $\omega_{FS}$ are globally well defined
	on $\mathbb{P}(\mathbb{C}^{n})\,.$\\

	We now want to relate the K\"{a}hler structure of $\mathbb{P}(\mathbb{C}^{n})$ with 
	the almost splitting Hermitian structure associated to the triple 
	$(\mathcal{P}_{n}^{\times},g_{F},\nabla^{(1)})$ discussed in \S\ref{section The geometry of}. 
	To this end, we set
	\begin{eqnarray}
		\mathbb{P}(\mathbb{C}^{n})^{\times}:=
		\big\{\,[z_{1},...,z_{n}]\in \mathbb{P}(\mathbb{C}^{n})\,\big\vert\,z_{i}\neq 0\,\,\,
		\textup{for all}\,\,\,i=1,...,n\,\big\}
	\end{eqnarray}
	and introduce the following smooth map 
	\begin{eqnarray}\label{equation definition tau}
		\tau\,:\,T\mathcal{P}_{n}^{\times}\rightarrow \mathbb{P}(\mathbb{C}^{n})^{\times}\,,
		\,\,\,[u]_{p}\mapsto \big[\,\sqrt{p_{1}}\,e^{iu_{1}/2},...,
		\sqrt{p_{n}}\,e^{iu_{n}/2}\,\big]\,.
	\end{eqnarray}
	The geometrical nature of the map $\tau$ is given by the following lemma. 
\begin{lemma}
	The map $\tau\,:\,T\mathcal{P}_{n}^{\times}\rightarrow \mathbb{P}(\mathbb{C}^{n})^{\times}$ 
		is a universal 
		covering map whose deck transformation group is a copy of $\mathbb{Z}^{n-1}\,.$
\end{lemma}
\begin{proof}
	Let $\mathbb{T}^{n-1}$ denotes the $(n{-}1)$-dimensional torus, and let us 
	consider the following diagram :
	\begin{eqnarray}
		\xymatrix@!C{
			T\mathcal{P}_{n}^{\times} 
			\ar[r]_{\displaystyle}^
			{\displaystyle\tau}
			\ar[d]_{j_{1}} & 
			\mathbb{P}(\mathbb{C}^{n})^{\times} \ar[d]_{j_{2}} 
			& 
			\\
			\mathcal{P}_{n}^{\times}\times\mathbb{R}^{n-1} 
			\ar[r]_{\displaystyle}^{\displaystyle\overline{\tau}}&
			\mathcal{P}_{n}^{\times}\times \mathbb{T}^{n-1}
				}
	\end{eqnarray}
	where
	\begin{description}
		\item[$\bullet$] $j_{1}([u]_{p}):=\big(p,(u_{1}-u_{n},...,u_{n-1}-u_{n})\big)\,,$
		\item[$\bullet$] $j_{2}\big(\big[\sqrt{p_{1}}e^{iu_{1}/2},...,
			\sqrt{p_{n}}e^{iu_{n}/2}\big]\big):= \big(p,(e^{i(u_{1}-u_{n})/2},...,
			e^{i(u_{n{-}1}-u_{n})/2})\big)\,,$
		\item[$\bullet$] $\overline{\tau}(p,u):=\big(p,(e^{iu_{1}/2},...,e^{iu_{n-1}/2})\big)\,.$
	\end{description}
	Clearly, $j_{1}$ and $j_{2}$ are diffeomorphims, and one easily sees that
	$\overline{\tau}$ is nothing but the quotient map associated to the (free and proper) 
	action of the group $\mathbb{Z}^{n-1}$ on $\mathcal{P}_{n}^{\times}\times
	\mathbb{R}^{n-1}$ given by $(k_{1},...,k_{n-1})\cdot
	(p,u):=\big(p,(u_{1}+4k_{1}\pi,...,u_{n-1}+4k_{n-1}\pi)\big)\,.$ As the diagram is manifestly 
	commutative, the lemma follows. 
\end{proof}
\begin{lemma}\label{lemma derivative de tau}
	For $([u]_{p},[v]_{p},[w]_{p})\in T_{[u]_{p}}T\mathcal{P}_{n}^{\times}\,,$
	we have
	\begin{eqnarray}
		&&(\phi_{z}\circ \tau)_{*_{[u]_{p}}}\big([u]_{p},[v]_{p},[w]_{p}\big)\nonumber\\
		&=&\Big(1/2\,\sqrt{p_{1}}\,e^{iu_{1}/2}(v_{1}+iw_{1}),...,
		1/2\,\sqrt{p_{1}}\,e^{iu_{n}/2}(v_{n}+iw_{n})\Big)\,,
	\end{eqnarray}
	where $z:=\big(\sqrt{p_{1}}\,e^{iu_{1}/2},...,\sqrt{p_{n}}\,e^{iu_{n}/2}\big)\in\mathbb{C}^{n}$
	and where $\phi_{z}\,:\,U_{z}\subseteq \mathbb{P}(\mathbb{C}^{n})
	\rightarrow [z]^{\perp}\subseteq \mathbb{C}^{n}$ is the chart 
	on $\mathbb{P}(\mathbb{C}^{n})$ introduced in \eqref{equation definition carte projective}.
\end{lemma}
\begin{proof}
	Let us fix $[u]_{p},[v]_{p},[w]_{p}\in 
	T\mathcal{P}_{n}^{\times}$ and set 
	\begin{eqnarray}
		z(t):=\big(\sqrt{p_{1}(t)}\,e^{i(u_{1}+tw_{1})/2},...,
		\sqrt{p_{n}(t)}\,e^{i(u_{n}+tw_{n})/2}\big)\in \mathbb{C}^{n}-\{0\}\,,
	\end{eqnarray}
	where $p(t)$ 
	is a smooth curve in $\mathcal{P}_{n}^{\times}$ satisfying $dp(t)/dt\big\vert_{0}=[v]_{p}\,.$
	Observe that $[z]=[z(0)]=\tau([u]_{p})\,,$ $\langle z(t),z(t)\rangle^{2}=1$
	and that $\langle z(t),\dot{z}(t)\rangle=0$ since $z(t)$ is normalized.\\
	Using the identification \eqref{equation dombrowski proba}, 
	Lemma \ref{lemme formule inverse identification} 
	as well as the formula $[\lambda\cdot z]=[z]$ ($\lambda\in \mathbb{C}-\{0\}$) which holds 
	on $\mathbb{P}(\mathbb{C}^{n})\,,$ we see that 
	\begin{eqnarray}\label{equation derivée de tau}
		&&(\phi_{z(0)}\circ \tau)_{*_{[u]_{p}}}\big([u]_{p},[v]_{p},[w]_{p}\big)
			=\dfrac{d}{dt}\bigg\vert_{0}(\phi_{z(0)}\circ \tau)
			\Big(\big[u{+}tw{-}E_{p(t)}(u{+}tw){\cdot} n\big]_{p(t)}\Big)\nonumber\\
		&=& \dfrac{d}{dt}\bigg\vert_{0}\,\phi_{z(0)}\Big(\big[\sqrt{p_{1}(t)}\,
			e^{i(u_{1}+tw_{1})/2}\,e^{-iE_{p(t)}(u+tw)/2 },...\nonumber\\
		&&\textbf{}\,\,\,\,\,\,\,\,\,\,\,\,\,\,\,\,\,\,\,\,\,\,\,\,\,\,\,\,\,\,\,\,\,\,\,\,\,\,\,\,\,
			\,\,\,\,\,...,\sqrt{p_{n}(t)}\,
			e^{i(u_{n}+tw_{n})/2}\,e^{-iE_{p(t)}(u+tw)/2 }\big]\Big)\nonumber\\
		&=&\dfrac{d}{dt}\bigg\vert_{0}\,\phi_{z(0)}\Big(\big[\sqrt{p_{1}(t)}\,
			e^{i(u_{1}+tw_{1})/2},...,\sqrt{p_{n}(t)}\,e^{i(u_{n}+tw_{n})/2}\big]\Big)\nonumber\\ 
		&=&\dfrac{d}{dt}\bigg\vert_{0}\,\phi_{z(0)}\big([z(t)]\big)=\dfrac{d}{dt}\bigg\vert_{0}\,
			\bigg(\dfrac{1}{\langle z(0),z(t)\rangle}\cdot z(t)-z(0)\bigg)\nonumber\\
		&=&\dfrac{-\langle z(0),\dot{z}(0) \rangle}{\langle z(0),z(0)\rangle^{2}}\,z(0)
			+\dfrac{1}{\langle z(0),z(0)\rangle}\,\dot{z}(0)=-\langle z(0),\dot{z}(0) \rangle\,z(0)
			+\dot{z}(0)=\dot{z}(0)\,.
	\end{eqnarray}
	The lemma is now a direct consequence of 
	\begin{eqnarray}\label{equation confirmation Japon!!!}
			\big(\dot{z}(0)\big)_{j}
			=\dfrac{d}{dt}\bigg\vert_{0}\,\sqrt{p_{j}(t)}\,e^{i(u_{j}+tw_{j})/2}&=&
			\dfrac{1}{2\,\sqrt{p_{j}}}\,p_{j}v_{j}\,e^{iu_{j}/2}+\sqrt{p_{j}}\,\dfrac{i}{2}w_{j}\,
			e^{iu_{j}/2}\nonumber\\
		&=&1/2\,\sqrt{p_{j}}\,e^{iu_{j}/2}(v_{j}+iw_{j})
	\end{eqnarray}
	(of course, in the above computations we use extensively the exponential representation 
	\eqref{equation derivée plus identification expo})\,.
\end{proof}
	For the next proposition (which is the main observation of this paper), recall that 
	$T\mathcal{P}_{n}^{\times}$ is endowed with its associated splitting Hermitian structure 
	$(G,J,\Omega)$ introduced at the end of \S\ref{section The geometry of}
	and that $\mathbb{P}(\mathbb{C}^{n})$ possesses its natural K\"{a}hler
	structure $(g_{FS},J_{FS},\omega_{FS})\,.$
\begin{proposition}\label{proposition origine des stat}
	The covering map $\tau\,:\,T\mathcal{P}_{n}^{\times}\rightarrow 
	\mathbb{P}(\mathbb{C}^{n})^{\times}$ defined in \eqref{equation definition tau} 
	has the following properties :
	\begin{eqnarray}\label{equation faut optimiser}
		\tau^{*}g_{FS}=G\,,\,\,\,\,\,\,\,\,\,\,
		\tau^{*}\omega_{FS}=\Omega\,,\,\,\,\,\,\,\,\,\,\,
		\tau_{*}\,J=J_{FS}\,\tau_{*}\,.
		\end{eqnarray}
\end{proposition}
\begin{proof}
	Let us fix 
	$[u]_{p},[v]_{p},[w]_{p},[\overline{v}]_{p},[\overline{w}]_{p}\in 
	T_{p}\mathcal{P}_{n}^{\times}$ and define the normalized vector 
	\begin{eqnarray}
	z:=\big(\sqrt{p_{1}}\,e^{iu_{1}/2},...,\sqrt{p_{n}}\,e^{iu_{n}/2}\big)\in\mathbb{C}^{n}\,.
	\end{eqnarray}
	According to Lemma \ref{lemma derivative de tau}, 
	we have :
	\begin{eqnarray}\label{equation je ne vois presque plus rien...}
		&&\Big\langle(\phi_{z}\circ \tau)_{*_{[u]_{p}}}([u]_{p},[v]_{p},[w]_{p}),
			(\phi_{z}\circ \tau)_{*_{[u]_{p}}}([u]_{p},[\overline{v}]_{p},
			[\overline{w}]_{p})\Big\rangle\nonumber\\
		&=&\Big\langle\Big(1/2\,\sqrt{p_{1}}\,e^{iu_{1}/2}(v_{1}+iw_{1}),...,
			1/2\,\sqrt{p_{1}}\,e^{iu_{n}/2}(v_{n}+iw_{n})\Big)\nonumber\\
		&&	,\Big(1/2\,\sqrt{p_{1}}\,e^{iu_{1}/2}(\overline{v}_{1}+i\overline{w}_{1}),...,
			1/2\,\sqrt{p_{1}}\,e^{iu_{n}/2}(\overline{v}_{n}+i\overline{w}_{n})\Big)\Big\rangle
			\nonumber\\
		&=& \sum_{j=1}^{n}\,\dfrac{1}{4}\,p_{j}\,(v_{j}-iw_{j})(\overline{v}_{j}
			+i\overline{w}_{j})= \sum_{j=1}^{n}\,\dfrac{1}{4}\,p_{j}\,\big(v_{j}\overline{v}_{j}
			+w_{j}\overline{w}_{j}+i(v_{j}\overline{w}_{j}-\overline{v}_{j}w_{j})\big)\nonumber\\
		&=& \dfrac{1}{4}\,\sum_{j=1}^{n}\,p_{j}\,\big(v_{j}\overline{v}_{j}
			+w_{j}\overline{w}_{j}\big)+\dfrac{i}{4}\,\sum_{j=1}^{n}\,p_{j}\,
			(v_{j}\overline{w}_{j}-\overline{v}_{j}w_{j})\nonumber\\
		&=& G_{[u]_{p}}\Big(([u]_{p},[v]_{p},[w]_{p}),([u]_{p},[\overline{v}]_{p},
			[\overline{w}]_{p})\Big)\nonumber\\
		&&\textbf{}\,\,\,\,\,\,\,+i\,\Omega_{[u]_{p}}\Big(([u]_{p},[v]_{p},[w]_{p}),([u]_{p},[\overline{v}]_{p},
			[\overline{w}]_{p})\Big)\,.
	\end{eqnarray}
	Comparing \eqref{equation je ne vois presque plus rien...} 
	with the definition of $g_{FS}$ and $\omega_{FS}$ given in 
		\eqref{equation definition de fubini-study} gives the first two relations in \eqref{equation faut optimiser},
%		we see that 
%	\begin{eqnarray}\label{equation faut que je prenne une douche}
%		\tau^{*}g_{FS}=G\,\,\,\,\,\textup{and}\,\,\,\,\,\tau^{*}\omega_{FS}=\Omega\,.
%	\end{eqnarray}
	and these two relations, together with the fact that both triples $\big(G,J,\Omega\big)$
	and $\big(g_{FS},J_{FS},\omega_{FS}\big)$ are compatible, imply the last relation in  \eqref{equation faut optimiser}. The proposition follows. 
\end{proof}
\section*{Acknowledgments} It is a pleasure to thank Hsiung Tze who pointed out to me the existence 
	of the geometrical formulation of quantum mechanics, and who, by his enthusiasm, gave 
	me the necessary motivation to study quantum mechanics outside of its usual presentation.\\
	This work was done with the financial support of the Japan Society for the Promotion 
	of Science. \\
%	Furthermore, I would like to express my gratitude to the SFB/TR
%	12, ``Symmetries and Universality in Mesoscopic systems", of the DFG for Hospitality. 

\end{document}